\documentclass[oneside, 11pt]{amsart}
\usepackage[]{fontenc}
\usepackage[latin9]{inputenc}
\usepackage{amsthm}
\usepackage{setspace}
\usepackage{amssymb}
\usepackage{euscript}
\usepackage{graphicx}

\makeatletter
\numberwithin{equation}{section} 
\numberwithin{figure}{section} 
\theoremstyle{plain}
\newtheorem{thm}{Theorem}
  \theoremstyle{plain}
  \newtheorem{prop}[thm]{Proposition}
  \theoremstyle{plain}
  \newtheorem{cor}[thm]{Corollary}

\DeclareMathOperator{\supp}{supp}
\DeclareMathOperator{\conv}{conv}
\DeclareMathOperator{\ext}{ext}

\def\fat{\scriptscriptstyle{\textup{fat}}}
\def\thin{\scriptscriptstyle{\textup{thin}}}

\makeatother

\begin{document}

\title[Marginals of shift-invariant measures]{On the finite-dimensional marginals of shift-invariant measures}

\author{J.-R. Chazottes}
\address[J.-R. Chazottes]{CPHT, CNRS-\'Ecole polytechnique, 91128 Palaiseau Cedex, France}
\email[J.-R. Chazottes]{jeanrene@cpht.polytechnique.fr}
\thanks{This work is part of the project `CrystalDyn' funded by the Agence Nationale de la Recherche. M. Hochman is supported by NSF grant 0901534. The authors acknowledge L. Sadun and Robert F. Williams for
allowing the reproduction of some figures from one of their papers.}
\author{J.-M. Gambaudo}
\address[J.-M. Gambaudo]{Laboratoire J. A. Dieudonn\'e, CNRS-Universit\'e de Nice-Sophia Antipolis,
Parc Valrose, 06108 Nice Cedex 02, France }
\email[J.-M. Gambaudo]{gambaudo@unice.fr}
\author{M. Hochman}
\address[M. Hochman]{Dept. of Mathematics,
Fine Hall, Washington Rd., Princeton, NJ 08544, USA }
\email[M. Hochman]{hochman@math.princeton.edu}
\author{E. Ugalde}
\address[E. Ugalde]{Instituto de F\'{\i}sica, Universidad Aut\'onoma de San Luis Potos\'{\i},
Av. Manuel Nava \#6, Zona Universitaria,
San Luis Potos\'{\i}, S.L.P., 78290 M\'exico}
\email[E. Ugalde]{ugalde@ifisica.uaslp.mx}

\begin{abstract}
Let $\Sigma$ be a finite alphabet, $\Omega=\Sigma^{\mathbb{Z}^{d}}$
equipped with the shift action, and $\mathcal{I}$ the simplex of shift-invariant measures on $\Omega$.
We study the relation between the restriction $\mathcal{I}_n$ of  $\mathcal{I}$ to the finite cubes $\{-n,\ldots,n\}^d\subset\mathbb{Z}^d$, and the polytope of ``locally invariant'' measures $\mathcal{I}_n^{loc}$. We are especially
interested in the geometry of the convex set $\mathcal{I}_n$ which turns out to
be strikingly different when $d=1$ and when $d\geq 2$. 
A major role is played by shifts of finite type which are naturally identified with faces of $\mathcal{I}_n$, and uniquely ergodic shifts of finite type, whose unique invariant measure  gives rise to extreme points of $\mathcal{I}_n$, although in dimension $d\geq 2$ there are also extreme points which arise in other ways.
We show that $\mathcal{I}_n=\mathcal{I}_n^{loc}$ when $d=1$, but in higher dimension they differ for $n$
large enough. We also show that while in dimension one $\mathcal{I}_n$ are polytopes with rational extreme points, in higher dimensions every computable convex set occurs as a rational image of a face of $\mathcal{I}_n$ for all large enough $n$.
\end{abstract}

\maketitle

\section{Introduction}

Let $\Sigma$ be a finite alphabet, and let $\Omega=\Sigma^{\mathbb{Z}^{d}}$
denote the full $d$-dimensional shift over $\Omega$.
$\Omega$ is compact and metrizable in the product topology, and the
group $\mathbb{Z}^{d}$ acts continuously on $\Omega$ by translation.
This action is denoted by $\Theta$ and the action of $u\in\mathbb{Z}^{d}$
by $\Theta^{u}$:
\[
(\Theta^{u}x)_{v}=x_{u+v}.
\]
The translation-invariant probability measures on $\Omega$ play an
important role in dynamical systems theory, probability and
thermodynamic formalism \cite{ruellebook},
but often these objects are subtle and difficult to describe. Therefore
constructions of examples and fine analysis of lattice models in physics
often proceed by studying appropriate measures on finite lattices
which converge, as the lattices grow, to a measure on the full shift.

Thus it is natural to study the relation between the simplex of measures
on finite lattices, and the restriction to the same lattices of the
simplex of invariant measures on the full shift. Intermediate between
them is a third object, namely, the polytope of measures which are
{}``locally invariant''. In this paper we make some contributions
to the understanding of these objects. We particularly focus on identifying
their extreme points, both because these determine the geometry of the
sets, and because in the set of invariant measures on the infinite lattice they correspond to the ergodic
measures, which are the most dynamically significant; and it is reasonable
to ask if there is a similar interpretation at the finite level.
Some previous work dealing with the problem of building a process from
a finite-dimensional marginal can be found in \cite{pivato} and overlaps
with our results in Theorem \ref{thm:phantom-extreme-points-in-d-2} and Theorem \ref{thm:faces-are-non-computable}.

Let us introduce some notation. The space of Borel probability measures
on $\Omega$ is denoted by $\mathcal{P}=\mathcal{P}(\Omega)$, and
the subset of measures invariant under the action $\Theta$ is denoted
$\mathcal{I}$. The space $\mathcal{P}$ (and hence $\mathcal{I}$)
carries the weak-{*} topology when measures are identified with bounded
linear functionals on $C(\Omega)$, and both spaces are closed and convex. The set
$\mathcal{I}$ is in fact a simplex, that is, each point in it has a unique representation
as the integral of a measure supported on its extreme points, and in fact the extreme points of $\mathcal{I}$ are precisely of the ergodic invariant measures.
$\mathcal{I}$ also has the remarkable feature that its extreme points
are dense. For proofs of these statements we refer to \cite{Georgii}.

Let
\[
\Lambda_{n}=\{-n,\ldots,n\}^{d}\subseteq\mathbb{Z}^{d}
\]
and let 
\[
\Omega_{n}=\Sigma^{\Lambda_{n}}.
\]
The finite-dimensional simplex of probability measures on $\Omega_n$
is denoted by $\mathcal{P}_{n}=\mathcal{P}(\Omega_{n})$. Let $\pi_{n}:\Omega\rightarrow\Omega_{n}$
denote the restriction operator $x\mapsto x|_{\Lambda_{n}}$, and
also denote by $\pi_{n}$ the induced map $\mathcal{P}\rightarrow\mathcal{P}_{n}$.
The image of $\mathcal{I}$ under this map is denoted $\mathcal{I}_{n}$.
Since $\mathcal{I}$ is compact and convex, and since $\pi_{n}$ is
continuous and linear, $\mathcal{I}_{n}$ is a closed convex subset of the finite-dimensional
vector space $\mathcal{P}_{n}$.

Another space of interest is the space $\mathcal{I}_{n}^{loc}\subseteq \mathcal{P}_{n}$ of
locally invariant measures. To define these,
we define a pattern to be a partial configuration $a:E\rightarrow\Sigma$,
for a finite set $E\subseteq\mathbb{Z}^{d}$. Translation $\Theta^{u}a$
of a pattern $a$ on $E$ is the pattern on $E-u$ defined by $(\Theta^{u}a)_{v}=a_{v+u}$.
The cylinder set $[a]\subseteq\Omega$ associated to $a\in\Sigma^{E}$
is
\[
[a]=\{x\in\Omega\,:\, x|_{E}=a\}
\]
and when $E\subseteq\Lambda_{n}$ we denote
\[
[a]_{n}=\{b\in\Omega_{n}\,:\, b|_{E}=a\}.
\]
Note that $[\Theta^{u}a]_{n}$ is defined for $a:E\rightarrow\Sigma$
as long as $E-u\subseteq\Lambda_{n}$.

A measure $\mu\in\mathcal{P}$ is invariant under $\Theta$ if $\mu(\Theta^{u}[a])=\mu([a])$
for every $u\in\mathbb{Z}^{d}$ and every pattern $a$. Similarly,
local invariance of $\mu\in\mathcal{P}_{n}$ means that for every
$E\subseteq\Lambda_{n}$, for every $a\in\Sigma^{E}$, and for every $u\in\mathbb{Z}^{d}$ such that
$E-u\subseteq\Lambda_{n}$, we have \begin{equation}
\mu([a]_{n})=\mu([\Theta^{u}a]_{n})\label{eq:local-invariance}\end{equation}
The set of locally invariant measures
on $\Omega_{n}$ is denoted $\mathcal{I}_{n}^{loc}$. Clearly $\mathcal{I}_{n}\subseteq\mathcal{I}_{n}^{loc}\subseteq\mathcal{P}_{n}$,
and each of these spaces is a closed, convex subset of the next. 

Recall that a polytope is the convex hull of a finite set. The intersection
of a polytope with an affine subspace is again a polytope. Note that as $E,a,n$ range over all their possible values, each
of the conditions \eqref{eq:local-invariance} amounts to intersecting
$\mathcal{P}_{n}$ with a linear subspace. If we choose as a basis for $\mathcal{P}_{n}$
the probability measures which give mass to a single cylinder set
$[a]_{n},$ $a\in\Omega_{n}$, then these subspaces are defined by
linear equations with integer coefficients. It follows that:
\begin{thm}
\label{thm:rational-extreme-points}$\mathcal{I}_{n}^{loc}$ is a
polytope whose extreme points are measures with rational range, that
is $\mu(A)\in\mathbb{Q}$ for all $A\subseteq\Omega_{n}$.
\end{thm}
The last statement holds independent of the dimension. However, the
nature of these extreme points differs quite drastically in dimension
1 and in higher dimensions. 
\begin{thm}
\label{thm:extreme-points-are-periodic-orbits-in-d-1}
For $d=1$, every extreme point of $\mathcal{I}_{n}^{loc}$ is the projection of
the uniform measure on a $\Theta$-periodic orbit. In particular, $\mathcal{I}_{n}=\mathcal{I}_{n}^{loc}$.
\end{thm}
Note however that $\mathcal{I}_{n}^{loc}$ is not a simplex, and measures
may not have a unique decomposition as convex combinations of extreme
points; thus in this sense the analogy of extreme points and ergodic
measures is false. Here is a simple example demonstrating this. Consider
the case $d=1$ and $n=1$, so we are considering measures on words of length
$|\Lambda_1|=3$. Let $\Sigma=\{1,2,3,4\}$. Let $\mu\in\mathcal{P}_{3}$ be the
uniform measure on sequences in which no symbol appears twice. Let
$S_{3}$ and $S_{4}$ denote the sets of sequences of length $3$
and $4$, respectively, in which no symbol appears twice. To each
$a\in S_{3}\cup S_{4}$ we can associate the $\Theta$-invariant measure $\mu_{a}\in\mathcal{I}$ supported
on the orbit of the periodic point $\ldots aaa\ldots\in\Omega$. One may
verify directly, or using Proposition \ref{pro:SFTs-and-faces} below,
that $\pi\mu_{a}$ is an extreme point in $\mathcal{I}_{1}$. It is then
clear that
\[
\mu=\frac{1}{|S_{3}|}\sum_{a\in S_{3}}\mu_{a}=\frac{1}{|S_{4}|}\sum_{b\in S_{4}}\mu_{b}.
\]
By replacing individual symbols by sequences of a fixed length which
have the property that any subsequence of a concatenation of them
has a unique parsing, we can construct, over a given alphabet, examples
like this for any $n$.

In contrast to Theorem \ref{thm:extreme-points-are-periodic-orbits-in-d-1},
in higher dimensions we have the following. 
\begin{thm}
\label{thm:phantom-extreme-points-in-d-2}For $d\geq2$ the map
$\pi_{n}:\mathcal{I}\rightarrow\mathcal{I}_{n}^{loc}$
is not onto. In particular there are extreme points of $\mathcal{I}_{n}^{loc}$
which do not correspond to invariant measures.
\end{thm}
The result was also proved by Pivato in \cite{pivato} using a result of Robinson \cite{RobinsonUndecidability}.

We may describe the relation between the sets $\mathcal{I}_{n}$
, $\mathcal{I}_{n}^{loc}$ and $\mathcal{I}$ is as follows. The restriction
maps $\pi_{n}$ are compatible in the sense that, if $\widetilde{\pi}_{n+1,n}:\mathcal{I}_{n+1}\rightarrow\mathcal{I}_{n}$
is again given by restriction, $a\mapsto a|_{\Lambda_{n}}$, then
\[
\pi_{n}=\widetilde{\pi}_{n+1,n}\pi_{n+1}.
\]
Also, the projections $(\pi_{n}(\mu))_{n=1}^{\infty}$ completely determine
$\mu\in\mathcal{P}$. Therefore $\mathcal{I}$ is the inverse limit
of the sets $\mathcal{I}_{n}$, and hence, in dimension $1$, the
inverse limit also of the  $\mathcal{I}_{n}^{loc}$. However,
for $d\geq2$ the situation is more subtle, due to the fact that the
maps $\pi_{n}$ are not onto $\mathcal{I}_{n}^{loc}$. Rather, we
have that following. For $k>n$ let $\widetilde{\pi}_{k,n}:\mathcal{I}_{k}^{loc}\rightarrow\mathcal{I}_{n}^{loc}$
be defined by restriction, as above. Then it is an easy consequence
of the definitions that 
\begin{prop}
\label{pro:intersection-of-local-invariant-measures}
$\mathcal{I}_{n}=\bigcap_{k=n}^{\infty}\widetilde{\pi}_{k,n}(\mathcal{I}_{n}^{loc})$.
\end{prop}
The new features of $\mathcal{I}_{n}$ which emerge in dimensions
$d\geq2$, as well as the tools used in their analysis, are closely
related to the dynamics of multidimensional shifts of finite type.
Recall that a closed, $\Theta$-invariant subset $X\subseteq\Omega$
is a shift of finite type (SFT) if there is an $n$ and a set of patterns
$L\subseteq\Omega_{n}$ such that $x\in X$ if and only if no pattern
from $L$ appears anywhere in $x$, that is, if $\pi_{n}(\Theta^{u}x)\notin L$
for every $u\in\mathbb{Z}^{d}$. In this case we write $X=SFT(L)$.
Given an SFT $X$ (or more generally any closed, $\Theta$-invariant
subset of $\Omega$) we denote by $\mathcal{I}(X)\subseteq\mathcal{I}$
the set of invariant measures supported on $X$. This is
a closed convex set.

A face of a convex set $C$ in a real vector space is the intersection of the boundary $\partial C$
with a supporting affine subspace. If this subspace is defined over the rationals, the face is said to be rational. A face consisting of one point is an exreme point, and the set of extreme points of $C$ is denoted $\ext C$. 
The relation between SFTs and $\mathcal{I}_{n}$
is given by the following simple but fundamental proposition. 
\begin{prop}
\label{pro:SFTs-and-faces}
If $X$ is an SFT defined by $L\subseteq\Omega_{n}$
then the set $\pi_{n}(\mathcal{I}(X))$ is a rational face of $\mathcal{I}_{n}$.
In particular, if $X$ is uniquely ergodic and $\mu$ is the unique
invariant measure on $X$, then $\pi_{n}(\mu)\in \ext\mathcal{I}_{n}$.
\end{prop}
Theorem \ref{thm:phantom-extreme-points-in-d-2} follows from Theorem
\ref{thm:rational-extreme-points}, the proposition above, and the
following:
\begin{thm}
\label{pro:SFT-with-irrational-masses}
For $d\geq2$ there exist uniquely
ergodic SFTs and two cylinder sets $C,C'$ such that the unique invariant
measure $\mu$ satisfies $\mu(C)/\mu(C')\notin\mathbb{Q}$.
\end{thm}
Alternatively there is an argument using undecidability of Wang's
tiling problem, see Section \ref{highdim}. 

As a geometric consequence, if $X$ is an SFT as in the proposition it follows that, for large enough $n$, the projection of its unique invariant measure to  $\mathcal{I}_n$ is located in the interior of a face of $\mathcal{I}_n^{loc}$, since the extreme points of the face containing the measure are rational.

We now turn to the structure of $\mathcal{I}_{n}$. While uniquely
ergodic SFTs give rise to extreme points, these do not exhaust the
possibilities:
\begin{thm}
\label{thm:non-ue-SFT-extreme-points}
For $d\geq2$ there are extreme points of
$\mathcal{I}_{n}$ which do not arise as restrictions of the unique
invariant measure on a uniquely ergodic SFT.
\end{thm}
Thus we do not have a complete desription of the extreme points. Nevertheless we can give some precise indication of the richness
of the sets $\mathcal{I}_{n}$. For this we require two definitions.

We say that a convex set $A\subseteq\mathbb{R}^{n}$ is \emph{effective}
if there is an algorithm which, on input $k$, computes the extreme points of a rational polytope polytope
$A_{k}$ which contains $A$, and such that $A=\bigcap_{k=1}^{\infty}A_{k}$.
This class, while countable, is very broad, including e.g. all rational
polytopes, the closed convex hulls of computable and bounded sequences of points,
etc. A stronger condition is that of computability: $A$ is \emph{computable}
if there is an algorithm which produces, for each $n$, a convex polytope
within Hausdorff distance $1/n$ of $A$. There are effective sets
which are not computable.

We say that a convex set $B\subseteq\mathbb{R}^{m}$ is a rational
image of a convex set $A\in\mathbb{R}^{n}$ if there is a rational
matrix and associated linear map $T:\mathbb{R}^{n}\rightarrow\mathbb{R}^{m}$
such that $TA=B$.

\begin{thm}
\label{thm:face-characterization}
Let $d\geq2$. Then $A$ is an effective
convex set if and only if for all large enough $n$, it is a rational
image of a rational face of $\mathcal{I}_{n}$.
\end{thm}
It is not hard to show that, for fixed $n$, the sets $\mathcal{I}_n$ are effective. On the other hand there exist effective non-computable convex sets, and rational images of computable
sets are computable. Thus:
\begin{thm}
\label{thm:faces-are-non-computable}
For $d\ge 2$ the sets $\mathcal{I}_{n}$ are
effective but, for large enough $n$, they are not computable.
\end{thm}
Non-computability of $\mathcal{I}_n$  was also established in \cite{pivato}.

These results should be compared with other propertis of multidimensional
symbolic dynamics which have emerged recently, in which the range
of certain dynamical parameters have been characterized in terms of
the level of computability \cite{HochmanMeyerovitch2010,Simpson2007,Hochman2009}.

Let us summarize the main points above. Proposition \ref{pro:SFTs-and-faces} shows that shifts of finite type
whose set of forbidden patterns is contained in $\Omega_n$
can naturally be identified with faces of $\mathcal{I}_n$. 
Moreover, for shifts of finite type carrying a unique invariant
measure, the corresponding face reduces to a single point
which is an extreme point of $\mathcal{I}_n$.
This fact holds irrespective of the dimension $d$ of the lattice. 
For $d=1$ this accounts for all of the extreme points, which arise as projections
of uniform measures on periodic orbits (Theorem \ref{thm:extreme-points-are-periodic-orbits-in-d-1}).
For $d\geq 2$ there still exist extreme points of $\mathcal{I}_n$
corresponding to projections of uniform measures on periodic orbits, but there are also many extreme points not of this kind. All these ``strange'' extreme points, as well as some of the extreme points arising from the invariant measures of uniquely ergodic SFTs, are properly contained in $\mathcal{I}_n^{loc}$, which is  a rational polytope.
As for the faces of $\mathcal{I}_n$ in dimension $d\geq 2$, we have characterized them up to rational images, but it remains an interesting open
problem to better understand the faces themselves.

\section{Proof of Theorem \ref{thm:rational-extreme-points}}

In this section we prove Theorem \ref{thm:rational-extreme-points}, which holds
in any dimension $d$.

Let $\mathcal{V}_{n}$ denote the space of real-valued functions $x:\Omega_{n}\rightarrow\mathbb{R}$,
$a\mapsto x_{a}$. Each $x\in\mathcal{V}_{n}$ may also be identified
with a signed measure, given by $x(A)=\sum_{a\in A}x_{a}$. With this
identification, we have
\[
\mathcal{P}_{n}=\{x\in\mathcal{V}_{n}\,:\,\sum x_{a}=1\mbox{ and }x_{a}\geq0\mbox{ for all }a\in\Omega_{n}\}.
\]
The extreme points are the vectors $\delta_a$ for $a\in \Omega_n$ and the faces of $\mathcal{P}_{n}$ are precisely the subspace of
$\mathcal{V}_{n}$ determined by setting $x_{a}=0$ for some $A\subseteq\Omega_{n}$.
Now, in order for $x\in\mathcal{P}_{n}$ to belong to $\mathcal{I}_{n}^{loc}$
it must additionally satisfy the condition in Equation \eqref{eq:local-invariance}
for each applicable choice of $E,u$ and $a\in\Sigma^{E}$, and this
condition is nothing other than the equation
\[
\sum_{b\in\Omega_{n}\,:\, b|_{E}=a}x_{b}=\sum_{c\in\Omega_{n}\,:\,(\Theta^{u}c)|_{E}=a}x_{c}
\]
which is an integer-valued linear equation in $\{x_{b}\}_{b\in\Omega_{n}}$.
Let $\mathcal{W}_{n}$ denote the solution set of these equations
for all parameters $E,u,a$. Then $x$ is an extreme point of $\mathcal{I}_{n}^{loc}$
if and only if $\{x\}$ is the intersection of $\mathcal{W}_{n}$
with some face of $\mathcal{P}_{n}$. Such an intersection is
a solution of a family of linear equations with integer coefficients. Such
equations have rational solutions if they have solutions at all, so if there is a unique solution it must be rational.
$\square$

\medskip

Another consequence of this proof is the following:
\begin{cor}
With $\mathcal{W}_{n}$ as in the proof of Theorem \ref{thm:rational-extreme-points},
the faces of $\mathcal{I}_{n}^{loc}$ are precisely the intersections
of $\mathcal{W}_{n}$ with faces of $\mathcal{P}_{n}$.
\end{cor}

\section{Proofs in dimension one}

We first prove a more general fact: For $d=1$, every $\mu\in\mathcal{I}_{n}^{loc}$
is the image under the restriction map $\pi_{n}:\mathcal{I}\rightarrow\mathcal{I}_{n}$
of a stationary $(n-1)$-step Markov measure $\widetilde{\mu}\in\mathcal{I}$.

To see this we construct the chain explicitly. For $\mu\in\mathcal{I}_{n}^{loc}$
we first construct a Markov chain with state space \[
V=\{a\in\Sigma^{[-n;n-1]}\,:\,\mu([a]_{n})>0\}\]
and given $a\in V$ let
\[
E_{a}=\{b\,:\, aa'=b'b\mbox{ for some }a',b'\in\Sigma\mbox{, and }\mu([aa']_{n})>0\}
\]
(here $aa'\in\Sigma^{[-n;n]}$ is the concatenation of $a$ and $a'$.
Note that $aa'\in\Omega_{n}$). Define the stochastic matrix $(p_{a,b})_{a,b\in V}$
by
\[
p_{a,b}=\frac{\mu([aa']_{n})}{\sum_{a''\in\Sigma}\mu([aa'']_{n})}
\]
whenever $b\in E_{a}$ and $aa'=b'b$ for $a',b'\in\Sigma$ (note that $a',b'$ are determined uniquely),
and set $p_{a,b}=0$ otherwise.

It is now straightforward to verify that the the probability vector
$p(a)=\mu([a]_{n})$, $a\in V$, is stationary for the stochastic
matrix $(p_{a,b})_{a,b\in V}$. By construction the shift-invariant
measure $\widetilde{\mu}$ on $V^{\mathbb{Z}}$, corresponding to this Markov chain, may be identified
with an invariant measure $\widetilde{\mu}$ on $\Sigma^{\mathbb{Z}}$
via the factor map $\tau:V^{\mathbb{Z}}\rightarrow\Sigma^{\mathbb{Z}}$
given by
\[
\tau(x)_{n}=(x_{n})_{0}.
\]
One verifies that the latter measure has marginal equal to
$\mu$. $\square$

\medskip

Let us make a few comments on this construction. First, the measure
$\widetilde{\mu}$ by construction gives positive measure to $[a]\subseteq \Sigma^\mathbb{Z}$ for $a\in V$,
and hence the matrix $(p_{a,b})$ is irreducible. Second, the transition
graph $G=(V,E)$, defined with $V$ as above and $E=\{(a,b)\;:\; a\in V\,,\, b\in E_{a}\}$,
is determined completely by the support of $\mu$. If $\nu\in\mathcal{I}_{n}^{loc}$
is a measure with $\supp\nu\subseteq\supp\mu$, then the graph constructed
as above for $\nu$ will be a sub-graph of the one constructed for
$\mu$. Consequently, $\widetilde{\nu}$ is supported on $\supp\widetilde{\mu}$,
although its support may actually be smaller. In particular this holds
whenever $\nu$ is the marginal on $\Omega_{n}$ of an invariant measure
on $\supp\widetilde{\mu}$. 
\begin{proof}
[Proof of Theorem \ref{thm:extreme-points-are-periodic-orbits-in-d-1}]
Let $\mu\in\mathcal{I}_{m}^{loc}$ be an extreme point. By the argument
above it is the image of an invariant Markov measure $\widetilde{\mu}$.
Also, since $\mu$ is an extreme point and $\mathcal{I}_{n}^{loc}=\mathcal{W}_{n}\cap\mathcal{P}_{n}$,
there is a face $\mathcal{F}$ of $\mathcal{P}_{n}$ such that $\mathcal{W}_{n}\cap\mathcal{F}=\{\mu\}$.
Here $\mathcal{W}_{n}\subseteq\mathcal{V}_{n}$ is the subspace defined
by the local invariance conditions as in the proof of Theorem \ref{thm:rational-extreme-points}.
We claim that $\mu$ must be the uniform measure on a periodic orbit,
or, equivalently, that the transition graph $G=(V,E)$ of $\widetilde{\mu}$
constructed as above consists of a single cycle. Indeed, otherwise
there would be a proper subgraph $G'\subseteq G$ which is itself a directed cycle. Taking any positive
stochastic matrix for $G'$, the corresponding Markov measure projects
down to a locally invariant measure $\nu\in\mathcal{I}_{n}^{loc}$.
Since $\supp\nu\subseteq\supp\mu$ we have \ $\nu\in\mathcal{F}$ (recall that faces in $\mathcal{P}_n$ are defined by setting some set of coordinates to zero),
so $\nu\in\mathcal{W}_{n}\cap\mathcal{F}=\{\mu\}$. On the other hand,
since $G'$ is a proper subgraph of $G$ there is some $A\subseteq\Omega_{n}$
with $\nu(A)=0$ and $\mu(A)>0$, a contradiction.
\end{proof}

\section{Proofs of higher dimensional results}\label{highdim}

We now turn to the higher dimensional case, $d\geq2$. Let us first
show that $\mathcal{I}_{n}^{loc}\neq\mathcal{I}_{n}$. We shall give
two different proofs of this fact. We begin with one based on Berger's
theorem \cite{Berger1966}, which asserts that there is no algorithm
which computes, for $L\subseteq\Omega_{n}$, whether $SFT(L)=\emptyset$.
\begin{proof}
[First proof of Theorem \ref{thm:phantom-extreme-points-in-d-2}] Suppose
that for every $n$ all of the extreme points of $\mathcal{I}_{n}^{loc}$
were projections of invariant measures. We show that we could
then decide, given $L\subseteq\Omega_{n}$, whether $SFT(L)=\emptyset$,
contradicting Berger's theorem.

Suppose $L\subseteq\Omega_{n}$ is given and define the convex hull
\[
\mathcal{F}=\conv\{\delta_{a}\,:\, a\in\Omega_{n}\setminus L\}.
\]
This is a face of $\mathcal{P}_{n}$, and $\mathcal{F}\cap\mathcal{W}_{n}$
is a face of $\mathcal{I}_{n}^{loc}$, where $\mathcal{W}_{n}$ is
the subspace defined by the local invariance conditions, as in the
proof of Theorem \ref{thm:rational-extreme-points}. Since $\mathcal{W}_{n}$
is given explicitly by integer linear equations, and $L$ is a given
finite set, we can decide if $\mathcal{F}\cap\mathcal{W}_{n}=\emptyset$.
In order to complete the proof, it is enough to show that $\mathcal{F}\cap\mathcal{W}_{n}=\emptyset$
if and only if $SFT(L)=\emptyset$.

In one direction, if $SFT(L)\neq\emptyset$ then there is an invariant
measure $\mu$ on $SFT(L)$, and by definition $\mu([a])=0$ for $a\in L$,
so also
\[
(\pi_{n}\mu)([a]_{n})=0.
\]
Thus $\pi_{n}\mu\in\mathcal{F}\cap\mathcal{W}_{n}$, so $\mathcal{F}\cap\mathcal{W}_{n}\neq\emptyset$.

Conversely, If $\nu\in\mathcal{F}\cap\mathcal{W}_{n}$ then we can
write $\nu=\sum p_{i}\nu_{i}$ where $p=(p_{i})$ is a probability
vector and $\nu_{i}\in\ext\mathcal{I}_{n}^{loc}$. In fact, clearly
$\nu_{i}\in\mathcal{F}\cap\mathcal{W}_{n}$ (this is true for geometric
reasons, but, more concretely, because $\nu([a]_{n})=0$ for $a\in L$
and therefore, since $\nu=\sum p_{i}\nu_{i}$, also $\nu_{i}([a]_{n})=0$
for all $i$ and $a\in L$. This implies by definition that $\nu_{i}\in\mathcal{F}$).
By assumption, $\nu_{i}=\pi_{n}\mu_{i}$ for some invariant measure $\mu_i$
on $\Omega$. But $\mu([a])=0$ for $a\in L$, so $\mu$ is supported
on $SFT(L)$, and therefore $SFT(L)\neq\emptyset$.
\end{proof}
Evidently, one cannot decide whether a locally invariant measure on
$\Omega_{n}$ can be extended to an invariant measure on $\Omega$.
This is the analogue of the undecidability of the extension problem for SFTs, i.e.\ that
one cannot decide, given a pattern $a\in\Sigma^{E}$, whether there
is a point $x\in SFT(L)$ such that $x|_{E}=a$. 

We now turn to $\mathcal{I}_{n}$. Since in dimension $1$ the sets
$\mathcal{I}_{n}$ and $\mathcal{I}_{n}^{loc}$ coincide, Theorem \ref{thm:extreme-points-are-periodic-orbits-in-d-1}
can be re-stated as follows: if $X\subseteq\Omega$ is a periodic
orbit then the unique invariant measure on it projects to an extreme
point of $\mathcal{I}_{n}$ for all large enough $n$. Here is the proof in the multi-dimensional case:

\begin{proof}
[Proof of Proposition \ref{pro:SFTs-and-faces}] Suppose $L\subseteq\Omega_{k}$
and $X=SFT(L)\neq\emptyset$. For $n\geq k$ let $\mathcal{F}\subseteq\mathcal{P}_{n}$
be the face defined by the condition $\mu([a]_{n})=0$ for $a\in L$.
Then, as in the proof of Theorem \ref{thm:phantom-extreme-points-in-d-2},
any $\mu\in\mathcal{F}\cap\mathcal{I}_{n}$ is the projection under
$\pi_{n}$ of a measure supported on $X$. Since $\mathcal{F}\cap\mathcal{I}_{n}$
is a face of $\mathcal{I}_{n}$, this proves the first part of the
proposition. The second follows from the fact that when $X$ is uniquely
ergodic there is only one measure in $\mathcal{I}$ supported on $X$
and therefore $\mathcal{F}\cap\mathcal{I}_{n}$ is a singleton (consisting
of the projection of this measure to $\mathcal{I}_{n}$), and hence
an extreme point. 
\end{proof}

Alternatively, the following proof gives some additional geometric information:

\begin{proof}
[Second proof of Theorem \ref{thm:phantom-extreme-points-in-d-2}] 
Assuming Theorem \ref{pro:SFT-with-irrational-masses}, there
is an SFT $X=SFT(L)$ which projects to an extreme point of $\mathcal{I}_{n}$
for all large enough $n$, but for which there are disjoint clopen sets with masses whose ratio is irrational, and therefore by Theorem \ref{thm:rational-extreme-points}
cannot be an extreme point of $\mathcal{I}_{n}^{loc}$.
\end{proof}

We offer two proofs of Theorem \ref{pro:SFT-with-irrational-masses}.

\begin{proof}
[First proof] 
One can apply \cite[Theorem 4.1]{HochmanMeyerovitch2010}. This uses the fact that there
are computable irrational numbers (e.g. there is an algorithm which
given $n$ outputs the $n$-th binary digit of $\sqrt{2}$), and requires
a few small modifications to the proof in  \cite[Theorem 4.1]{HochmanMeyerovitch2010} in order to verify that the resulting
SFT is uniquely ergodic, but these are straightforward given the machinery
in \cite{Mozes1989}.
\end{proof}
\begin{proof}
[Second proof]
The second proof relies on two facts. First, there are translation-minimal
spaces of geometric tilings (i.e.\ tilings of the Euclidean plane by
translates and rotations of finitely many polygons) in which the ratios
of the areas of prototiles is irrational. An example of such a tiling
is the Penrose kites-and-darts tiling system \cite{Robinson1996}.
Second, there is an orbit equivalence between this system and the
suspension of an appropriate uniquely ergodic $\mathbb{Z}^{2}$-SFT
\cite{SadunWilliams2003}. This gives appropriately weighted cylinder
sets. More details are provided in Section \ref{tilings}.
\end{proof}

\section{Rational images of $\mathcal{I}_{n}$}

In this section we prove Theorems \ref{thm:non-ue-SFT-extreme-points}
and \ref{thm:face-characterization}. We rely on some results about
effective subdynamics of SFTs which appeared in \cite{Hochman2009} and were
 generalized in \cite{AubrunSablik2010} and \cite{DurandRomashchenkoShen2009} (the earlier in \cite{HochmanMeyerovitch2010} can easily be modified as
 well to give the same results). We describe these results next.

 We begin with some definitions about one-dimensional patterns, which
we call words. Below we shall use superscrips to indicate the dimensions of patterns and probability spaces on patterns. Thus $\Omega_n^1$ is the space of 1-dimensional patterns over $\Sigma$, with shape $\Lambda_n^1=[-n,n]\cap\mathbb{Z}$, and $\Omega_n^2$ is the space of two-dimensional patterns over $\Lambda_n^2=[-n,n]^2\cap\mathbb{Z}$. We similarly write 
$\mathcal{P}_n^1$ and $\mathcal{P}_n^2$ for the space of probability measures on $\Omega_n^1,\Omega_n^2$, respectively.

The set of words over an alphabet $\Sigma_{0}$ is
$\Sigma_{0}^{*}=\bigcup_{n=0}^{\infty}\Sigma_0^n$ . Consider an alphabet $\Sigma_{0}$ and a family $L\subseteq\Sigma^{*}$.
The subshift $X\subseteq\Sigma^{\mathbb{Z}}$ consisting of all sequences
which omit the patterns from $L$ is a subshift which we shall say
is of infinite type and denote $SIT(L)$. If we write $L_{n}=L\cap\bigcup_{i=1}^{n}\Omega_{n}$
then $SIT(L)$ is equal to
\[
SIT(L)=\bigcap_{n=1}^{\infty}SFT(L_{n}).
\]
Note that in fact every subshift arises in this way. 

A subshift $X$ is \emph{effective }if there is a recursively enumerable
set $L$ such that $X=SIT(L)$; that is, $L$ should have the property
that there exists an enumeration $L=\{a_{1},a_{2},\ldots\}$ and an
algorithm which on input $n$ outputs $a_{n}$.

Effective subshifts form a much larger class than SFTs, but there
is a close relation. Given $X\subseteq\Sigma^{\mathbb{Z}}$ let $\widehat{X}\subseteq\Sigma^{\mathbb{Z}^{2}}$
denote the subshift
\[
\widehat{X}=\{\widehat{x}\in\Sigma^{\mathbb{Z}^{2}}\,:\,\widehat{x}|_{\mathbb{Z}\times\{0\}}\in X\mbox{ and }\widehat{x}_{i,j}=\widehat{x}_{i,j+1}\mbox{ for all }i,j\in\mathbb{N}\}
\]
that is, $\widehat{X}$ is obtained from $X$ be extending each configuration
$x\in X$ vertically to a 2-dimensional configuration in which one
sees $x_{i}$ in the $i$-th column.
\begin{thm}
[\cite{Hochman2009,DurandRomashchenkoShen2009,AubrunSablik2010}]\label{thm:subdynamics}
If $X\subseteq\Sigma^{\mathbb{Z}}$ is an effective subshift
then there is a 2-dimensional SFT $Y$ which factors onto $\widehat{X}$.
Furthermore $Y$ can be chosen so that it has entropy $0$ and has
no periodic points.
\end{thm}
Note that the additional properties of $Y$ ensure that $Y$ can be realized as a subshift over any alphabet
$\Sigma$, in particular we may assume that $Y\subseteq\Omega^{2}$
for the fixed alphabet $\Sigma$ with which we are working. 

Given a subshift $Z$ let $\mathcal{I}(Z)$ denote the convex set
of invariant measures on $Z$. We denote by $\Delta$ an unspecified finite alphabet, possibly different from $\Sigma$.
\begin{prop}
Let $X\subseteq\Delta^{\mathbb{Z}}$ be an effective subshift. Then
$\pi_{n}(\mathcal{}(X))\subseteq\mathcal{P}_{n}^{1}$ is a rational
image of a rational face of $\mathcal{P}_{n}^{2}$ for all large enough
$n$.\end{prop}
\begin{proof}
Let $X$ be given and let $Y\subseteq\Omega^{2}$ and $\varphi:Y\rightarrow\widehat{X}$
be as in the theorem. By Theorem \ref{thm:subdynamics}, $\pi_{n}\mathcal{I}(Y)\subseteq\mathcal{P}_{n}^{2}$
is a face of $\mathcal{P}_{n}^{2}$ for large enough $n$. The factor
map $\varphi:Y\rightarrow\widehat{X}$ is given by a partition $\mathcal{A}=\{A_{1},\ldots,A_{s}\}$
of $\Omega$ into clopen sets and the rule that $\varphi(x)_{j}=i$
if and only if $\Theta^{u}x\in A_{i}$. Since each $A_{i}$ is the
union of finitely many cylinder sets, there is an $n_{0}$ such that
all $A_{i}$ depend on coordinates of modulus $<n_{0}$. It follows
that $\mathcal{P}_{n}^{2}(\widehat{X})$ is a rational image of the
face $\pi_{n+n_{0}}\mathcal{I}(Y)\subseteq\mathcal{P}_{n+n_{0}}^{2}$,
for all large enough $n$. But $\mathcal{P}_{n}^{1}(X)$ is a rational
image of $\mathcal{P}_{n}^{2}(\widehat{X})$, since it is obtained
by projecting to the subspace determined by the coordinates $\{(i,0)\,:\,-n\leq i\leq n\}$.
This proves the claim.
\end{proof}
Using this, Theorem \ref{thm:face-characterization} follows from the next proposition.
\begin{prop}
If $C\subseteq\mathbb{R}^{k}$ is an effective convex set then there
is an alphabet $\Sigma$, an effective subshift $X\subseteq\Sigma^{\mathbb{Z}}$,
and an $n$ such that $C$ is a rational image of $\pi_{n}\mathcal{I}(X)\subseteq\mathcal{P}_{n}^{1}$
for all large enough $n$.\end{prop}
\begin{proof}
By assumption there is an algorithm which produces a sequence $C_{n}$
of rational polytopes descending to $C$. We may assume that $C\subseteq\Delta$,
where $\Delta$ is the $m$-dimensional simplex
\[
\Delta=\Big\{\underline{p}\in\mathbb{R}^{m+1}\;,\;\sum p_{i}=1\mbox{ and }p_{i}\geq0\mbox{ for all }i\Big\}.
\]
Indeed, we can certainly assume that $C\subseteq[0,1/2m]^{m}$ simply
by scaling by an appropriate rational and translating by a rational
vector (in fact a suitable scaling factor and translation vector can
be computed from the first approximation $C_{1}$, since $C\subseteq C_{1}$).
Now consider the set $C\times\mathbb{R}\subseteq\mathbb{R}^{m}\times\mathbb{R}\cong\mathbb{R}^{m+1}$,
which is clearly effective. Then, since $\Delta$ is effective, the
set $C'=(C\times\mathbb{R})\cap\Delta$ is effective. Furthermore
the projection $\mathbb{R}^{m+1}\rightarrow\mathbb{R}^{m}$ onto the
first $m$ coordinates, which is rational, maps $C'$ to $C$. Thus
if $C'$ is a rational image of a rational face of $\mathcal{P}_{n}^{1}$
then so is $C$.

Henceforth we assume that $C\subseteq\Delta$. 

Let $\Sigma_{0}=\{1,\ldots,m+1\}$, and identify $\mathcal{P}_{0}^{1}$
with $\Delta$ in the obvious manner. We shall construct an effective
subshift $X\subseteq\Sigma_{0}^{\mathbb{Z}}$ such that the collection
of probability vectors $\mathcal{I}(X)\cap\mathcal{P}_{0}^{1}$, consisting
the the one-letter marginals of invariant measures on $X$, is exactly
$C$. 

We introduce some notation. If $E\subseteq\Sigma_{0}^{N}$ is a set
of words of length $N$, let $X_{E}\subseteq\Sigma_{0}^{\mathbb{Z}}$
denote the set of all bi-infinite sequences which are concatenations
of words from $E$. This is a subshift. For $a\in E$ and $i\in\Sigma_{0}$
let
\[
p_{i}(a)=\#\{1\leq j\leq N\,:\, a_{j}=i\}.
\]
These are the empirical frequencies of symbols in $a$, and write
$p(a)=(p_{1}(a),\ldots,p_{m+1}(a))$. Notice that if $C'$ is a closed
convex set and $p(a)\in C'$ for all $a\in E$, then $\pi_{0}\mu\in C'$
for all $\mu\in\mathcal{I}(X_{E})$. Also note that $\pi_{0}^{-1}(C')$
is a closed convex subset of $\mathcal{I}(X_{E})$.

The effectiveness condition means that there is an algorithm computing,
for each $n\in\mathbb{N}$, a vector $v_{n}\in\mathbb{Q}^{m+1}$ and
$r_{n}\in\mathbb{Q}$, such that $x\in C$ if and only if $x\cdot v_{n}\leq r_{n}$
for all $n$, and such that
\[
C_{n}=\{x\,:\, x\cdot v_{i}\leq r_{i}\mbox{ for }1\leq i\leq n\}.
\]
Let $n_{0}$ denote the smallest integer such that $C_{n_{0}}$ is
bounded (we may assume that such an integer exists since $C$ is bounded). Notice
that the extreme points of $C_{n}$ are rational, and can be found
by solving a system of linear equations with rational coefficients,
so are they computable. 

We construct next a sequence $N_{n}\in\mathbb{N}$ and subsets $E_{n}\subseteq\Sigma_{0}^{N_{n}}$
recursively, for $n\geq n_{0}$, satisfying the following properties:
\begin{enumerate}
\item Every $a\in E_{n+1}$ is a concatenation of words from $E_{n}$ (and
in particular $N_{n}|N_{n+1}$).
\item $C_{n}=\conv\{p(a)\,:\, a\in E_{n}\}$.
\end{enumerate}
To begin the recursion, let $n=n_{0}$. Note that the first condition
is vacuous. To satisfy the second take $N_{n_{0}}$ to be large enough
so that the extreme points of $C_{n_{0}}$ can be written with common denominator
$N_{n_{0}}$, and let $E_{n_{0}}$ be the set of words $a\in\Sigma_{0}^{n_{0}}$
such that $p(a)\in\ext(C_{n_{0}})$.

Assuming we have defined $N_{n}$ and $E_{n}$, define $N_{n+1}$
and $E_{n+1}$ as follows. Since $C_{n+1}\subseteq C_{n}$, every
extreme point of $C_{n+1}$ is a convex combination of the extreme
points of $C_{n}$, and since all these points are rational, the coefficients
of these convex combination are rational as well. Choose $N'_{n+1}$
to be a common denominator for all of these weights and let $N_{n+1}=N_{n}N'_{n+1}$.
Let $E_{n+1}$ denote all words $a\in\Sigma_{0}^{N_{n+1}}$ which
are concatenations of words from $E_{n}$ and satisfy $p(a)\in\ext C_{n+1}$.
Every extreme point of $C_{n+1}$ is represented at least once by
our choice of $N_{n+1}$.

It is clear that $X_{E_{n+1}}\subseteq X_{E_{n}}$, and we obtain
a subshift
\[
X=\bigcap_{n=1}^{\infty}X_{E_{n}}.
\]
We also have $\mathcal{I}(X_{E_{n+1}})\subseteq\mathcal{I}(X_{E_{n}})$,
and one easily sees that $\mathcal{I}(X)=\bigcap\mathcal{I}(X_{E_{n}})$.
We next claim that $\pi_{0}\mathcal{I}(X_{E_{n}})=C_{n}$. Indeed,
the inclusion `$\subseteq$' is clear, since if $b$ is a concatenation of blocks
from $E_{n}$ then clearly $p(b)\in C_{n}$ (because the frequency
of each block is in $C_{n}$, and $C_{n}$ is convex); and so by the
ergodic theorem $\pi_{0}\mu\in C_{n}$ for all $\mu\in\mathcal{I}(X_{E_{n}})$,
so $\pi_{0}\mathcal{I}(X_{E_{n}})\subseteq C_{n}$. On the other hand
the extreme points of $C_{n}$ are of the form $p(a)$ for some $a\in E_{n}$,
and if $\mu$ is the unique invariant measure on the periodic point
$\ldots aaa\ldots\in X_{E_{n}}$ clearly $\pi_{0}\mu=p(a)$, so
$\ext C_{n}\subseteq\pi_{0}\mathcal{I}(X_{E_{n}})$
and by convexity $C_{n}\subseteq\pi_{0}\mathcal{I}(X_{E_{n}})$. Thus
we also have relation $C_{n}\subseteq\pi_{0}\mathcal{I}(X_{E_{n}})$,
and consequently $C=\pi_{0}\mathcal{I}(X)$ (note however that we
do not claim this projection is 1-1. In general there will be many
invariant measures on $X$ with a given vector $p$ for the marginal
frequencies).

Finally, the construction of $N_{n}$ and $E_{n}$ is explicit, and
the resulting sequences are recursively enumerable. All that remains
is to show that $X$ is effective. But clearly each $X_{E_{n}}$ is
effective, since for example if we let
\[
L_{n}=\Omega_{0}^{3N_{n}}\setminus\{b'a'a''b''\in\Sigma_{0}^{3N_{n}}\,:\, a',a''\in E_{n}\,,\, b',b''\in\Sigma_{0}^{*}\}
\]
then $X_{E_{n}}=SFT(L_{n})$. Thus
\[
X=\bigcap SFT(L_{n})=SIT\big(\bigcup L_{n}\big).
\]
But $\bigcup L_{n}$ is recursively enumerable, since the languages
$L_{n}$ are finite and the sequence of languages $L_{n}$ is computable.
\end{proof}
This completes the proof of Theorem \ref{thm:face-characterization}.

Finally, Theorem \ref{thm:non-ue-SFT-extreme-points} follows from the following observation. There are effective convex sets with uncountably many points, e.g. the unit disc. These must be rational images of rational faces of $\mathcal{P}_n^2$ for large enough $n$. This implies that $\mathcal{P}_n^2$ has uncountably many extreme points for large enough $n$, and these cannot all arise as projections of invariant measures on uniquely ergodic SFTs, since there are only countably many SFTs. Hence some do not arise in this way.

\section{A proof of Theorem \ref{pro:SFT-with-irrational-masses} based on Penrose tilings}\label{tilings}

In this section we give a second proof of Theorem \ref{pro:SFT-with-irrational-masses} in the
case $d=2$. We first introduce some background and notations and we  refer the reader to \cite{Robinsonsurvey} for a survey on tilings of $\mathbb{R}^d$ and symbolic dynamics.

\subsection{Tiling dynamical systems}

Consider a set of prototiles $P=\{p_1,\ldots,p_{\scriptscriptstyle{|\Sigma|}}\}$ labelled by a finite alphabet $\Sigma$.
By prototile we mean a polyhedron in $\mathbb{R}^d$. A tiling of $\mathbb{R}^d$ constructed with $P$
is a countable set of polyhedra $(t_i)_{i\in\mathbb{Z}}$ labeled by $\Sigma$ such that:
\begin{itemize}
\item[-] The union of the tiles $t_i$'s is $\mathbb{R}^d$;
\item[-] Whenever $i\neq j$, the interiors of $t_i$ and $t_j$ are disjoint;
\item[-] Whenever $t_i\cap t_j\neq \emptyset$ and $i\neq j$, $t_i$ and $t_j$ share a full $(d-1)$-face;
\item[-] For any $i$, there exists $j(i)\in\{1,\ldots,|\Sigma|\}$ and $u_i\in\mathbb{R}^d$ such that
$t_i=p_{j(i)}+u_i$ and the label of $t_i$ is the same as that of $p_{j(i)}$.
\end{itemize}

We also assume that our tiling systems are locally finite: adjacent prototiles can meet in only a finite number of ways in any tiling of $\mathbb{R}^d$. 

Let $\mathcal{T}(P)$ be the set of all tilings constructed with $P$, and let $\bar\Theta$ denote the associated action of $\mathbb{R}^d$ by translation:
$$
\bar\Theta^u(\{t_i\})=\{t_i+u\},\;\textup{for}\;\{t_i\}\in \mathcal{T}(P).
$$
There is a natural metrizable topology for which $\mathcal{T}(P)$
is compact and the action  by translation $\bar\Theta$ is continuous. Colloquially two tiles are close whenever they coincide in a large ball around
the origin, up to a small translation (see  \cite{Robinsonsurvey} for more details about this construction). The pair $(\mathcal{T}(P),\bar\Theta)$ is called a tiling dynamical system.

\subsection{Penrose tiling system}

Consider the `thin' and `fat' triangles displayed in Figure \ref{penrosetriangles} \footnote{It is customary
to use arrowheads to indicate adjecency rules. Each triangle can be represented as a polyhedron by replacing the arrowheads by appropriate dents and bumps to fit the general definition of tilings given above.}. Together with their rotation by multiples of $2\pi/10$, they generate
a set of prototiles $\EuScript{P}$ with 40 elements ($|\Sigma|=40$) which generate the Penrose system $\mathcal{T}(\EuScript{P})$.
Remarkably these triangles also allow to make Penrose tilings by substitution by appropriate inflation rules \cite{GSbook}.
\begin{figure}[h!]
\centering
\includegraphics[width=8cm]{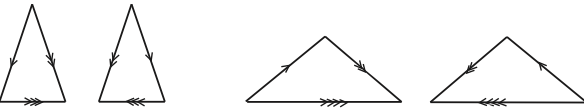}
\caption{The tiles of the Penrose tiling.}
\label{penrosetriangles}
\end{figure}
The dynamical system $(\mathcal{T}(\EuScript{P}),\bar\Theta)$ associated with Penrose tiles system is minimal and uniquely ergodic (see for instance  \cite{Robinson1996}). Actually much more can be said about the ergodic properties of Penrose dynamical systems.\newline
We recall that a van Hove sequence for $\mathbb{R}^2$ is a sequence $\{F_n\}_{n\geq 1}$ of
bounded measurable subsets of $\mathbb{R}^2$ satisfying
$$
\lim_{n\to\infty} \lambda\big(\big(\partial F_n\big)^{+r}\big)\big/ \lambda(F_n)=0,\; \textup{for all}\;r>0
$$
where, for a subset $A$ of $\mathbb{R}^2$, $A^{+r}:=\big\{x\in\mathbb{R}^2 : \textup{dist}(x,A)\leq r\big\}$.

\begin{prop}
Let $\lambda$ be the Lebesgue measure in $\mathbb{R}^{2}$ and consider a van Hove sequence
$\{F_n\}_{n\geq 1}$ in $\mathbb{R}^{2}$. For any  tiling $T$ in $\mathcal{T}(\EuScript{P})$,
any $x$ in $\mathbb{R}^{2}$ and any $n\geq 1$,  let $L^{\fat}_T(x+ F_n)$
(resp.  $L^{\thin}_T(x+ F_n)$) be the number of fat (resp. thin) triangles of
$T$ in  $x+ F_n$. Then the following limits
$$
\nu(\textup{fat}) = \lim_{n\to \infty}\frac{L^{\fat}_T(x+ F_n)}{\lambda(F_n)}
\quad \textup{and}\quad\nu(\textup{thin}) = \lim_{n\to \infty}\frac{L^{\thin}_T(x+ F_n)}{\lambda(F_n)}
$$
exist, are independent of $T$ in $\mathcal{T}(\EuScript{P})$, $x$ in $\mathbb{R}^{2}$
and of the van Hove sequence  $\{F_n\}_{n\geq 1}$, and satisfy:
$$
\frac{\nu(\textup{fat})}{\nu(\textup{thin})}=\frac{1+\sqrt 5}{2}.
$$
\end{prop}

The fact that the limits exist and are independent of the tiling, the reference point and the van Hove sequence, is a direct consequence of unique ergodicity as shown in \cite[Theorem 2.7]{LMS}. The fact that the ratio of these limits is the golden mean can be easily checked by choosing the van Hove sequence made with the series of successive inflations of a given triangle with respect to the substitution rules.

\subsection{Wang tilings}

A special class of tilings is that of Wang tilings, which we describe for $d=2$ for the sake
of simplicity. A Wang prototile is a unit square with sides parallel to the axes, and with colored
edges which restrict adjecency: abutting edges of adjacent tiles must have the same color.\footnote{
The colors can of course be replaced by dents and bumps to fit the above definition of prototiles.}
Consider a finite collection $W$ of Wang prototiles; it gives rise to the a tiling space consisting of all admissible tilings of the plane, i.e. tilings in which the adjecency rules are obeyed. Wang tilings
are closely tied to shifts of finite type because, giving the tiling space, we can translate every tiling so that the vertices of its tiles lie on
the lattice $\mathbb{Z}^2$. The resulting family of tilings can be identified with a collection of configurations in $W^{\mathbb{Z}^2}$, which is easily seen to be an SFT. Thus the tiling space is (conjugate to) the suspension of a shift of finite type $X_W$. This basic fact and its
reciprocal was noticed in \cite{Klaus}. Let us also mention again Berger's theorem:
there is no algorithm which computes for a given $W$ whether the corresponding SFT
is empty or, equivalently, whether $\mathcal{T}(W)=\emptyset$.

The importance of Wang tilings stems from the following result proved by Sadun and
Williams \cite{SadunWilliams2003} which we formulate here for $d=2$, although
it is valid for any dimension:  Let $\mathcal{P}$ be a finite set of prototiles and  $\mathcal{T}(P)$ the associated tiling space. Then there exists a finite set of Wang prototiles $W$ and a homeomorphism $h$ which realizes
an orbit equivalence between the dynamical systems $(\mathcal{T}(P),\bar\Theta)$
and $(\mathcal{T}(W),\bar\Theta)$. In particular, there exists a finite alphabet
$W$ such that the tiling space $\mathcal{T}(P)$ is homeomorphic to
the suspension of the $\mathbb{Z}^2$-SFT $X_W$.

\subsection{From the Penrose tiling system to a uniquely ergodic SFT }

An SFT can be explicitly constructed from the Penrose system, as shown in
\cite{SadunWilliams2003}. 
For the sake of convenience, we sketch the construction given therein in full details.
The first step is to construct a rational tiling space from the original Penrose system by a homeomorphism.
By a rational tiling we mean a tiling whose edge vectors all have rational coordinates. This amounts to solving finitely many
equations with integral coefficients yielding the result by elementary linear algebra. By rescaling
if necessary, we can assume that the obtained rational Penrose tiling is in fact
integral and that each triangle contains at least one unit square whose corners have integer coordinates. Figure \ref{rationalpenrose} displays a patch of the integral Penrose tiling.

\begin{figure}[h!]
\centering
\includegraphics[width=5cm]{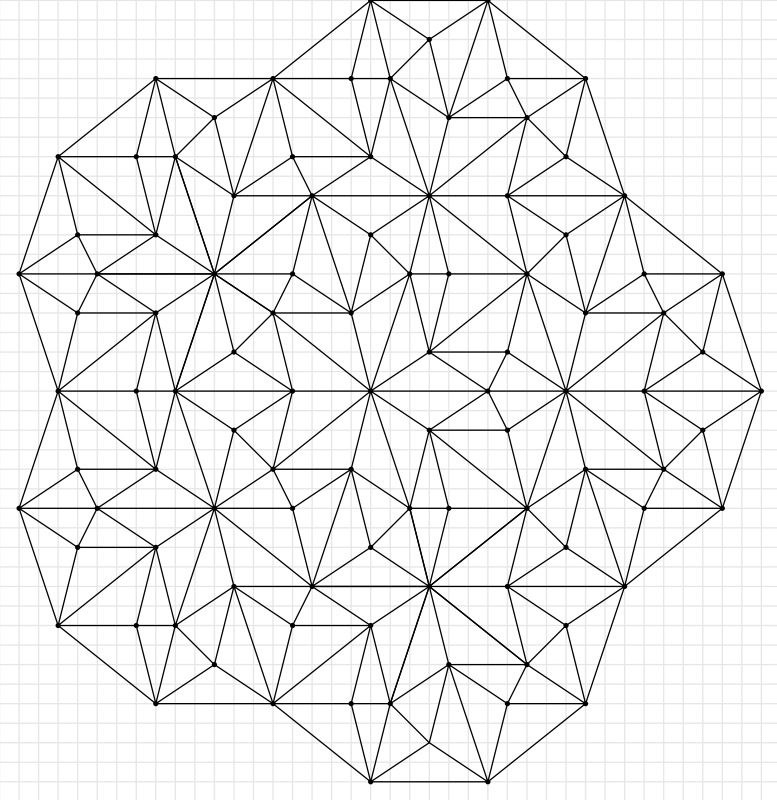}
\caption{A patch of the rational Penrose tiling.}
\label{rationalpenrose}
\end{figure}

Let us denote by  $\EuScript{P}_{\textup{int}}$ the new set of 40 triangles obtained this way and let us keep on calling thin  and fat triangles the respective images of the thin and fat Penrose triangles. Clearly the new dynamical system  $(\mathcal{T}(\EuScript{P}_{\textup{int}}),\bar\Theta)$ is orbit equivalent to $(\mathcal{T}(\EuScript{P}),\bar\Theta)$. The following proposition is a direct consequence of the fact that the homeomorphism that maps a Penrose tiling $T$  on a integral Penrose tiling $T_{\textup{int}}$, maps van Hove sequences on van Hove sequences.
\begin{prop}
\label{Montana}
The new  dynamical system $(\mathcal{T}(\EuScript{P}_{\textup{int}}),\bar\Theta)$
is again minimal and uniquely ergodic.
Furtheremore, for any  tiling $T_{\textup{int}}$ in $\mathcal{T}(\EuScript{P}_{\textup{int}})$,
any $x$ in $\mathbb{R}^{2}$ and any $n\geq 1$,   the following limits
$$
\nu_{int}(\textup{fat})
=\lim_{n\to \infty}\frac{L^{\fat}_{T_{int}}(x+ F_n)}{\lambda(F_n)}
\quad \textup{and}\quad\nu_{\textup{int}}(\textup{thin}) =
\lim_{n\to \infty}\frac{L^{\thin}_{T_{\textup{int}}}(x+ F_n)}{\lambda(F_n)}
$$
exist, are independent on ${T_{\textup{int}}}$ in $\mathcal{T}(\EuScript{P}_{\textup{int}})$, $x$ in $\mathbb{R}^{2}$ and on the van Hove sequence  $\{F_n\}_{n\geq 1}$, and satisfy again:

$$\frac{\nu_{\textup{int}}(\textup{fat})}{\nu_{\textup{int}}(\textup{thin})}\, =\, \frac{1+\sqrt 5}{2}.$$

\end{prop}
The next step is to transform the rational Penrose tiling into a Wang tiling.  The basic idea is to replace the straight edges of the triangles
with zig-zags, that is with sequences of unit displacements in the coordinates directions. The
next figure shows how the previous patch is transformed.

\begin{figure}[h!]
\centering
\includegraphics[width=5cm]{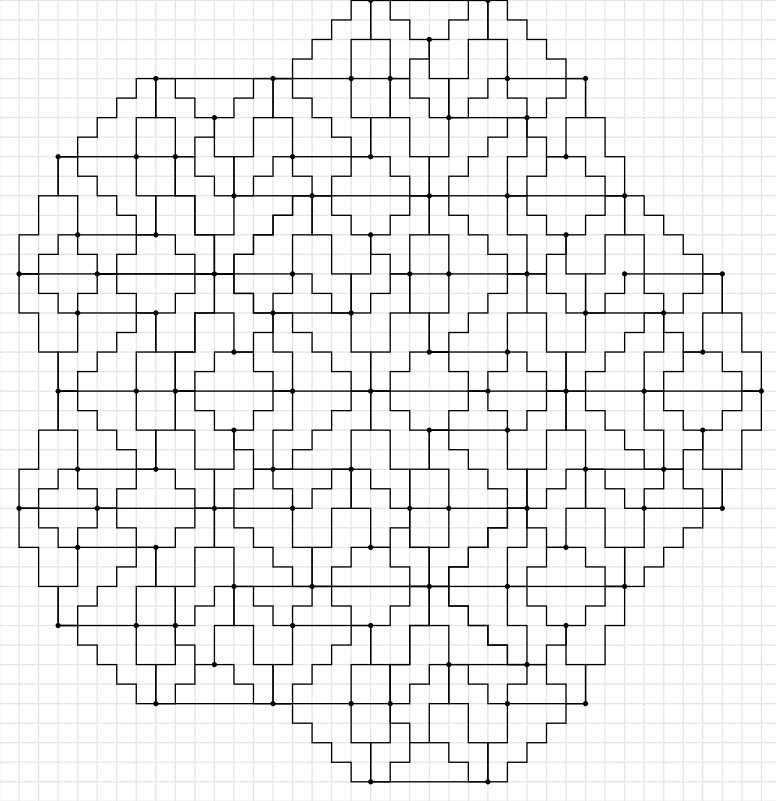}
\caption{A patch of the square Penrose tiling.}
\label{penrosecarrele}
\end{figure}

It remains to put appropriate colors on the edges of the squares to obtain a Wang tiling.
Colors are used to encode each of the $40$ prototiles and to encode the
matching rules between them.  Let $W$ be the set of $\vert W\vert$ Wang tiles constructed this way. Clearly the two dynamical systems $(\mathcal{T}(\EuScript{P}_{int}),\bar\Theta)$ and $(\mathcal{T}(W),\bar\Theta)$ are conjugate and thus  $(\mathcal{T}(W),\bar\Theta)$ is minimal and uniquely ergodic. In each triangle of
$\EuScript{P}_{\textup{int}}$ choose one Wang tile.  We get this way a collection of 20 distinct Wang tiles chosen in fat triangles: $W^{\fat} = \{w^{\fat}_1, \dots, w^{\fat}_{20}\}$ and $20$ distinct Wang tiles chosen in thin triangles:
$W^{\thin}= \{w^{\thin}_1, \dots, w^{\thin}_{20}\}$. Let $\mu$ be the unique translation invariant measure of the associated SFT $X_W$. As a by-product of Proposition \ref{Montana}, we get:
$$
\frac{\sum_{i=1}^{i =20}\mu([w^{\fat}_i])}{\sum_{i=1}^{i =20}\mu([w^{\thin}_i])}\, =\, \frac{1+\sqrt 5}{2}.
$$
Consequently there exists a pair $(i, j)$ in $\{1, \dots , 20\}^2$ such that 
$$
\frac{\mu([w^{\fat}_i])}{\mu([w^{\thin}_j])}\, \notin \mathbb{Q}.
$$

Thus we have construct a uniquely ergodic SFT on a finite alphabet with $\vert W\vert $ letters for which the ratio of the measures 2 cylinders with size 1 is irrational. This shows Theorem  \ref{pro:SFT-with-irrational-masses}  for an alphabet with $\vert W\vert $ letters.  By replacing individual symbols in $W$  by square words of fixed size written in any given finite alphabet with the property that 
that any subsequence of a concatenation of them
has a unique parsing, we can show Theorem  \ref{pro:SFT-with-irrational-masses}   for any finite alphabet. 

\bibliographystyle{plain}
\bibliography{bib}

\end{document}